\documentclass[10]{amsart}

\newtheorem{theorem}{Theorem}[section]
\newtheorem{lemma}[theorem]{Lemma}

\newtheorem{corollary}[theorem]{Corollary}

\newtheorem{problem}[theorem]{Problem}

\newtheorem{algorithm}[theorem]{Algorithm}

\theoremstyle{definition}

\newtheorem{example}[theorem]{Example}
\newtheorem{remark}[theorem]{Remark}

\usepackage{amscd,amssymb}

\begin{document}

\title[Weak central polynomials for multiplications of simple algebras]
{Weak central polynomials for algebras\\
of multiplications of simple algebras}

\author[Vesselin Drensky and Mikhail Zaicev]
{Vesselin Drensky and Mikhail Zaicev}
\address{Institute of Mathematics and Informatics,
Bulgarian Academy of Sciences,
Acad. G. Bonchev Str., Block 8,
1113 Sofia, Bulgaria}
\email{drensky@math.bas.bg}
\address{Department of Algebra, Faculty of Mathematics and Mechanics, Moscow State University, Moscow 119992, Russia}
\address{Moscow Center of Fundamental and Applied Mathematics, Moscow, 119991, Russia}
\email{zaicevmv@mail.ru}

\thanks{The second named author was partially supported by the Russian Science Foundation grant 26-11-00001.}



\begin{abstract} We give a simple proof for the existence of weak central polynomials for the algebra of
multiplications of a finite-dimensional simple (non-associative) algebra. As an example we present explicit weak central polynomials
in the cases of the three-dimensional simple Lie algebra and the Jordan algebra of the two-dimensional vector space with non-degenerate symmetric bilinear form.
\end{abstract}

\maketitle

\section{Introduction}
Let $R$ be an associative algebra over an arbitrary field $F$ and let $V$ be a vector subspace of $R$ which generates $R$ as an algebra.
A polynomial $f(x_1,\ldots,x_n)$ in the free associative algebra $F\langle X\rangle=F\langle x_1,x_2,\ldots\rangle$ is an identity for the pair
$(R,V)$ if $f(v_1,\ldots,v_n)=0$ in $R$ for all $v_1,\ldots,v_n\in V$. Similarly, if the center of $R$ is non-trivial, $c(x_1,\ldots,x_n)\in F\langle X\rangle$ is
a (non-trivial or essential) central polynomial for the pair $(R,V)$ if $c(v_1,\ldots,v_n)$ belongs to the center of $R$ for all $v_1,\ldots,v_n\in V$
and $c(x_1,\ldots,x_n)$ is not an identity for $(R,V)$.

In a talk given in 1956 Kaplansky \cite{K1} (see also the revised version from 1970 \cite{K2}) asked several problems which
motivated significant research activity. One of the problems was for the existence of central polynomials for
$k\times k$ associative matrix algebras $M_k(F)$ ($k>2$). Below we survey some results on central polynomials.

Latyshev and Shmelkin \cite{LS} constructed a central (clearly non-homogeneous) polynomial in one variable
for the matrix algebra $M_k(F)$ over a finite field $F$.
The first central polynomials over an arbitrary ground field were constructed by Formanek \cite{F1} and Razmyslov \cite{R1}
and this gave rise to a serious revision of the theory of algebras with polynomial identities, see e.g. the books by
Procesi \cite{P}, Jacobson \cite{J2}, Rowen \cite{Ro}, Formanek \cite{F4}, Drensky and Formanek \cite{DF}, Giambruno and Zaicev \cite{GZ}, Kanel-Belov and Rowen \cite{K-BRo},
and Kanel-Belov, Karasik and Rowen \cite{K-BKRo}.
In particular, important theorems on PI-algebras were established or simplified using central polynomials.
The construction of Formanek \cite{F1} will be used later to illustrate our approach.
The central polynomial of Razmyslov \cite{R1} is multilinear and also may be used for our purposes.
Later Kharchenko \cite{Kh} gave a short proof for the existence of central polynomials for $M_k(F)$ using classical results of Amitsur \cite{A1, A2}.

In the sequel central polynomials with different additional properties were constructed by several authors:
Halpin \cite{H},  Drensky and Kasparian \cite{DK}, Bondari \cite{Bo},
Drensky and Piacentini-Cattaneo \cite{DPC}, Drensky \cite{D1}, Giambruno and Valenti \cite{GV}.
Confirming a conjecture of Regev \cite{Re}, Formanek \cite{F3} showed that a polynomial in two sets of $k^2$ skew-symmetric variables
is a central polynomial for $M_k(F)$. The existence of such a polynomial has important consequences for the study
of the sequence of cocharacters of matrices, see \cite{F2}.
The proof of Kharchenko \cite{Kh} is not constructive. There are several other nonconstructive proofs:
by Braun \cite{Bra} and Bre\v{s}ar \cite{Bre1, Bre2}.

Going back to the problem for existence of central polynomials for pairs $(R,V)$, Razmyslov \cite{R2} gave an explicit construction of central polynomials
for irreducible finite-dimensional representations of semisimple Lie algebras. Giambruno, Shestakov, and Zaicev \cite{GSZ}
used the existence of central polynomials of special kind to study the asymptotics of the codimensions
for several classes of algebras including simple algebras with a special non-degenerate form,
finite-dimensional Jordan or alternative algebras and many more. We refer to the monograph of Giambruno and Zaicev \cite{GZ} for more
applications of central polynomials in the quantitative study of polynomial identities.

It is easy to see that if the algebra $R$ has multilinear (i.e., linear in each variable) central polynomials, then
any pair $(R,V)$ also has central polynomials. Starting with an explicitly given central polynomial for $R$, we present a simple algorithm to produce
a central polynomial for $(R,V)$. A slight modification of the proof works also for any (also not multilinear) central polynomial of $R$.
We apply this argument to the case when $A$ is a finite-dimensional simple (non-associative) algebra, $\text{End}_F(A)$
is the algebra of $F$-vector space endomorphisms of $A$, $V\subseteq \text{End}_F(A)$ is the vector space of the multiplications of $A$
and $R={\mathcal M}(A)\subseteq \text{End}_F(A)$ is the algebra of the multiplications of $A$ (which is generated by $V$ as a subalgebra of $\text{End}_F(A)$).
The same arguments hold when $\rho:L\to \text{End}_F(W)$ is any non-trivial finite-dimensional irreducible representation of the simple Lie algebra $L$
and $V=\rho(L)\subseteq\text{End}_F(W)$.
As an example we start with the central polynomial of Formanek \cite{F1} for $3\times 3$ matrices
(of degree 9) over an infinite field $F$ of characteristic different from 2
and present an explicit central polynomial of degree 10 for two algebras of multiplications.
The first one is when $V$ is the vector space of the multiplications of the three-dimensional simple Lie algebra $L_3$.
Equivalently, $V=\text{ad}(L_3)\subset \text{End}_F(L_3)=R$.
In characteristic 0 the complete description of the weak central polynomials in this case is given by Domokos and Drensky \cite{DoD}.
In particular, it is known that the standard polynomial $s_3(x_1,x_2,x_3)$ of degree 3 is a weak central polynomial.
The second algebra is when $V$ is the vector space of the multiplications of the Jordan algebra $A=F+W_2$ of the two-dimensional vector space $W_2$
with non-degenerate symmetric bilinear form. In this case we have found also an essential central polynomial of degree 8.

\section{The algorithm}
In this section we fix an arbitrary field $F$, an associative $F$-algebra $R$ with a nontrivial center and a vector subspace $V\subseteq R$
which generates $R$ as an algebra. If necessary, we shall enlarge (or replace) the set of free generators of the free associative algebra
$F\langle X\rangle=F\langle x_1,x_2,\ldots\rangle$ with new variables $Y=\{y_{ij}\}$ and shall work in $F\langle X,Y\rangle$ (or in $F\langle Y\rangle$).
The following easy statement plays an essential role in our considerations.

\begin{lemma}\label{the main lemma}
Let $c(X)=c(x_1,\ldots,x_n)\in F\langle X\rangle$ be a non-trivial multilinear central polynomial of the algebra $R$.
Then there exists a polynomial $c'(Y)\in F\langle Y\rangle$ which is a non-trivial central polynomial for the pair $(R,V)$.
\end{lemma}

\begin{proof}
Let $V$ have a vector space basis $U=\{u_1,u_2,\ldots\}$. Since $V$ generates $R$ as an algebra,
$R$ is spanned on the products
\[
W=\{w_i=u_{i_1}\cdots u_{i_h}\mid u_{i_j}\in U,h=1,2,\ldots\}.
\]
Let us fix a subset $W_0$ of $W$ which is a basis of the vector space $R$.
Since $c(X)$ is not a polynomial identity for $R$ and is multilinear, there are elements $w^{(1)},\ldots,w^{(n)}$ from $W_0$ such that
$c(w^{(1)},\ldots,w^{(n)})\not=0$. Let
\[
w^{(p)}=u_{i_1}^{(p)}\cdots u_{i_{h_p}}^{(p)},\quad u_{i_j}^{(p)}\in U.
\]
We define
\[
c'(Y)=c(y_{11}\cdots y_{1h_1},\ldots,y_{n1}\cdots y_{nh_n})\in F\langle Y\rangle.
\]
Evaluated on $V$, the products $y_{j1}\cdots y_{jh_j}$ belong to $R$. Hence $c'(Y)$ is a central element evaluated on $V$, i.e., $c'(Y)$
is a central polynomial for the pair $(R,V)$. It is non-trivial because
\[
c(u_{i_1}^{(1)}\cdots u_{i_{h_1}}^{(1)},\ldots,u_{i_1}^{(n)}\cdots u_{i_{h_n}}^{(n)})=c(w^{(1)},\ldots,w^{(n)})\not=0.
\]
\end{proof}

\begin{remark}\label{any central polynomial}
If the central polynomial $c(X)$ for $R$ is not multilinear we also may produce a central polynomial for $(R,V)$.
There exist $r_1,\ldots,r_n\in R$ such that $c(r_1,\ldots,r_n)\not=0$. Since $r_1,\ldots,r_n$ can be expressed as
\[
r_p=\sum\alpha_{pi}v_{i_1}^{(p)}\cdots v_{i_{h_i}}^{(p)},\quad v_{ij}^{(p)}\in V,\alpha_{pi}\in F,p=1,\ldots,n,
\]
the polynomial
\[
c'(Y)=c(\sum\alpha_{1i}y_{1i_1}\cdots y_{1i_{h_i}},\ldots,\sum\alpha_{ni}y_{ni_1}\cdots y_{ni_{h_i}})\in F\langle Y\rangle
\]
does not vanish on $V$ and hence is a non-trivial central polynomial for the pair $(R,V)$.
\end{remark}

When the base field $F$ is constructive and the algebra $R$ is finite-dimensional Lemma \ref{the main lemma} gives an algorithm
with input a multilinear central polynomial for $R$ and output a central polynomial for the pair $(R,V)$.

\begin{algorithm}\label{construction of central polynomial}
Let $\dim R=d$, $\dim V=m$, and let $c(x_1,\ldots,x_n)\in F\langle X\rangle$ be a multilinear central polynomial for $R$.
Fix a basis $U=\{u_1,\ldots,u_m\}$ of $V$. Let
\[
c'_m(Y)=c(y_{11}\cdots y_{1h_1},\ldots,y_{n1}\cdots y_{nh_n}),
\]
\[
h=(h_1,\ldots,h_n), \quad 1\le h_p\le d-m+1,\quad p=1,\ldots,n.
\]
For every $n$-tuple $h=(h_1,\ldots,h_n)$ evaluate $c'_h(Y)$ on all collections $\{u_{i_{p_j}}\}$, where
$u_{i_{pj}}\in U$, $j=1,\ldots,h_p$, $p=1,\ldots,n$.
There exists $c'_h(Y)$ and a collection $\{u_{i_{p_j}}\}$ such that
\[
c'_h(u_{i_{pj}})=c(u_{i_{11}}\cdots u_{i_{1h_1}},\ldots,u_{i_{n1}}\cdots u_{i_{nh_n}})\not=0.
\]
Then $c'_h(Y)$ is a non-trivial central polynomial for the pair $(R,V)$.
\end{algorithm}

\begin{proof} Since $V$ generates $R$ as an algebra and $R$ is finite-dimensional, there exists a positive integer $q$ such that
\[
V\subsetneqq V+V^2\subsetneqq\cdots\subsetneqq V+V^2+\cdots+V^q=V+V^2+\cdots+V^{q+1}=R,
\]
where $V^i$ is spanned on all products $u_{j_1}\cdots u_{j_i}$, $i=1,\ldots,q$. Clearly,
\[
q\leq \dim R-\dim V+1=d-m+1.
\]
Since $c(X)$ is a non-trivial multilinear central polynomial for $R$, there exist products $u_{i_{p1}}\cdots u_{i_{ph_p}}$, $p=1,\ldots,n$,
such that $c'_h(u_{i_{pj}})\not=0$. Then $c'_h(Y)$ is a non-trivial central polynomial for the pair $(R,V)$.
\end{proof}

The arguments in the proof of the above algorithm give the following theorem.

\begin{theorem}\label{main theorem}
Let $c(x_1,\ldots,x_n)$ be a multilinear central polynomial for the algebra $R$.
If $\dim(R)=d$ and $\dim(V)=m$, then there is a weak central polynomial $c'(x_1,\ldots,x_N)$
for the pair $(R,V)$ such that $1\leq \deg(c')\leq n(d-m+1)$.
\end{theorem}

\begin{remark}\label{Remark 1}
As it was commented in \cite{D2}, every central polynomial of the algebra $R$ is either a weak central polynomial or a weak polynomial identity of the pair $(R,V)$.
The main difficulty in Algorithm \ref{construction of central polynomial} is to find a central polynomial of $R$ which does not vanish evaluated on $V$.
On the other hand the algorithm and Theorem \ref{main theorem} have the disadvantage that they do not give weak central polynomials
for the pair $(R,V)$ which are not central polynomials for $R$.
\end{remark}

In the paper \cite{R2} Razmyslov have proved the following statement.

\begin{theorem}\label{Razm}
Suppose that a semisimple finite-dimensional Lie algebra $L$ over an
algebraically closed field of characteristic zero has dimension $m$. Suppose that an enveloping
algebra $R$ of $L$ is simple and has center distinct from zero. For some natural
number $k$ there then exists a central polynomial of the pair $(R,L)$ of degree $km$.
\end{theorem}

Now we are able to clarify this result and to give an upper bound for the degree of the central polynomial.

\begin{corollary}\label{upper bound}
Under assumptions of Theorem \ref{Razm} there exists a central polynomial $c'$ of the pair $(R,L)$, where $\dim(R)=d$, $\dim(L)=m$, such that:

{\rm (i)} If $F$ is an arbitrary algebraically closed field of any characteristic, then $\deg(c')\leq 2d(d-m+1)$;

{\rm (ii)} If $\text{\rm char}(F)=0$ or $\text{\rm char}(F)>d-\sqrt{d}$, then $\deg(c')\leq d(d-m+1)$;

{\rm (iii)} If $\text{\rm char}(F)=0$ and $d\geq 9$, then $\deg(c')\leq (\sqrt{d}-1)^2+4)(d-m+1)$.

{\rm (iv)} If $\text{\rm char}(F)\not=2$ and $d=4$, then $\deg(c')=2$.
\end{corollary}

\begin{proof}
(i) Since $F$ is algebraically closed, $R$ is isomorphic to the matrix algebra $M_t(F)$, $d=t^2$.
According to \cite{F3} for an arbitrary field of any characteristic $R$ has a multilinear central polynomial
$c(x_1,\ldots,x_n)$ of degree $n=2t^2=2d$. Now this case of our corollary follows from Theorem
\ref{main theorem}.

(ii) The polynomial $c(x,y_1,\ldots,y_t)$ of Formanek \cite{F1} is central of degree $d=\dim M_t(F)=t^2$. It is multilinear in the variables
$y_1,\ldots,y_t$ and of degree $t^2-t$ in the variable $x$. If $\text{char}(F)=0$ or $\text{char}(F)>\deg_x(c)$, then the linearization of $c$ in $x$
is also an essential central polynomial of $M_t(F)$. This case follows again from Theorem \ref{main theorem}.

(iii) For $t\geq 3$ (i.e. for $d\geq 9$) the central polynomial for $M_t(F)$ constructed in \cite{D1} is of degree $(t-1)^2+4=(\sqrt{d}-1)^2+4$
and again we apply Theorem \ref{main theorem}.

(iv) If $d=4$, then the only possibility is $L\cong sl_2(F)$ and $R=M_2(F)$. By the theorem of Razmyslov \cite{R0} the pair $(M_2(F),sl_2(F))$ satisfies the identity
$[x^2,y]$ and $x^2$ is an essential central polynomial for the pair.
\end{proof}

\section{Applications}
In this section $A$ is a finite-dimensional simple (non-associative) algebra, $V\subseteq \text{End}_F(A)$ is the vector space of the multiplications of $A$
and $R$ is the subalgebra of $\text{End}_F(A)$ generated by $V$. The following lemma is similar to \cite[Proposition 1.1]{PoS}.

\begin{lemma}\label{multiplication algebra of simple algebra}
Let $A$ be a finite-dimensional simple algebra. Then the algebra ${\mathcal M}(A)$ of its multiplications
is isomorphic to a matrix algebra $M_k(Q)$ over a finite-dimensional division algebra $Q$.
\end{lemma}

\begin{proof}
The multiplications of $A$ act on the vector space $A$ and $A$ has a structure of an ${\mathcal M}(A)$-module.
Since $A$ is simple, $A$ is a faithful simple ${\mathcal M}(A)$-module, i.e. ${\mathcal M}(A)$ is a primitive $F$-algebra.
By a well known theorem in ring theory \cite[page 33, Theorem 3]{J1}
a finite-dimensional primitive ring is isomorphic to a matrix algebra $M_k(Q)$ where $Q$ is a finite-dimensional division algebra.
\end{proof}

\begin{corollary}\label{central polynomials for multiplication algebra}
Let $A$ be a finite-dimensional simple algebra and let its algebra of multiplications ${\mathcal M}(A)$ be of dimension $t^2$.
Then ${\mathcal M}(A)$ has the same multilinear central polynomials
as the matrix algebra $M_t(F)$.
\end{corollary}

\begin{proof}
It is well known, see e.g. \cite[Theorem I]{W} or \cite[page 27, Theorem 1.6.17]{J3},
that if $Q$ is a division algebra of dimension $p^2$, then the matrix algebra $M_k(Q)$ has a splitting field, i.e. an extension $E$ of the ground field $F$
such that $E\otimes_FM_k(Q)\cong M_{kp}(E)$. Let $t=kp$ and let $c(x_1,\ldots,x_n)$ be a multilinear essential central polynomial for $M_t(E)=M_{kp}(E)$. The evaluations of $c(x_1,\ldots,x_n)$ on $M_t(E)$
are linear combinations with coefficients in $E$ of the evaluations of $c(x_1,\ldots,x_n)$ on $M_k(Q)$.
Hence some of the evaluations of $c(x_1,\ldots,x_n)$ in $M_k(Q)$ are different from 0 and $c(x_1,\ldots,x_n)$ is an essential central polynomial also for $M_k(Q)$.
Similarly, since $c(x_1,\ldots,x_n)\in F\langle X\rangle$, the evaluations of $c(x_1,\ldots,x_n)$ in $M_t(E)$ are linear combinations with coefficients in $E$ of its evaluations in $M_t(F)$.
Hence ${\mathcal M}(A)$ has the same multilinear central polynomials in $F\langle X\rangle$ as $M_t(F)$.
\end{proof}

\begin{lemma}\label{dimensions of multiplication}
Let $A$ be a finite-dimensional simple (non-associative) algebra and let ${\mathcal L}(A)$ and ${\mathcal R}(A)$ be, respectively, the vector spaces of the left and ring multiplications of $A$.
Then $\dim_F(A)\leq \dim_F({\mathcal L}(A))+\dim_F({\mathcal R}(A))\leq 2\dim_F(A)$.
\end{lemma}

\begin{proof}
Let $\ell(a)$ and $r(a)$ be the operators of left and right multiplications by $a\in A$. The mapping $a\to\ell(a)$, $a\in A$, defines a vector space homomorphism $A\to {\mathcal L}(A)$
and hence $\dim({\mathcal L}(A))\leq\dim_F(A)$. A similar inequality holds for the dimension of the vector space ${\mathcal R}(A)$.
Hence $\dim_F({\mathcal L}(A))+\dim_F({\mathcal R}(A))\leq 2\dim_F(A)$. Now, consider the mapping $\mu:A\to ({\mathcal L}(A),{\mathcal R}(A))$
defined by $\mu(a)=(\ell(a),r(a))$, $a\in A$. If the dimension of the image $\mu(A)$ is less than $\dim(A)$, then there is a nonzero element $a\in A$ such that $\mu(a)=(0,0)$.
Hence the element $a$ generates a proper ideal of $A$ which contradicts with the simplicity of $A$. Hence
\[
\dim_F(A)=\dim(\mu(A))\leq \dim_F({\mathcal L}(A))+\dim_F({\mathcal R}(A)).
\]
\end{proof}

\begin{example}\label{lower bound for dimension}
In the proof of Lemma \ref{dimensions of multiplication} we used that $\dim({\mathcal L}(A))\leq\dim_F(A)$. It is easy to construct a simple algebra with the property that
$\dim({\mathcal L}(A))<\dim_F(A)$ and $\dim({\mathcal R}(A))<\dim_F(A)$. Let $A$ be the algebra with basis $\{a,b,c\}$ and multiplications defined by
\[
ab=c,\,ac=b,\,bc=a
\]
and all other products between the basis elements are equal to 0. If
\[
u=\alpha a+\beta b+\gamma c\not=0,\, \alpha,\beta,\gamma\in F,
\]
then
$au=\gamma b+\beta c$, $(au)c=\gamma a$, $((au)c)c=\gamma b$, $((au)c)b=\gamma c$.
Hence, if $\gamma\not=0$, then the two-sided ideal $I$ generated by $u$ coincides with $A$. Similarly, if $\gamma=0$, then
$au=\beta c$, $b(au)=\beta a$, $a(b(au))=\beta b$
and if $\beta \not=0$, then again $I=A$. The case $\beta=\gamma=0$ is considered similarly: since $u=\alpha a\not=0$, then
$uc=\alpha b$, $ub=\alpha c$ and $I=A$.
On the other hand, the vector subspaces $\dim({\mathcal L}(A))$ and $\dim({\mathcal R}(A))$ of $\text{End}_F(A)\cong M_3(F)$ are spanned, respectively, by the homomorphisms defined by
\[
\ell(u):(a,b,c)\to(0,\alpha c,\beta a+\alpha c),\, r(u):(a,b,c)\to(\gamma b+\beta c,\gamma c,0),
\]
$u=\alpha a+\beta b+\gamma c$, $\alpha,\beta,\gamma\in F$, i.e. $\dim_F({\mathcal L}(A))= \dim_F({\mathcal R}(A))=2<\dim_F(A)=3$.
\end{example}

The following theorem is the main consequence of Theorem \ref{main theorem}.

\begin{theorem}\label{main corollary}
Let $A$ be a finite-dimensional simple (non-associative) algebra, $\dim_F(A)=m$, let $V$ be the vector space of the multiplications of $A$
and let $R={\mathcal M}(A)$ be the algebra of multiplications of $A$.
Let $c(x_1,\ldots,x_n)$ be a multilinear central polynomial for the algebra ${\mathcal M}(A)$.
Then the pair $(R,V)$ has a weak central polynomial of degree $\leq n(4m^2-m+1)$.
If $A$ is a Lie or a Jordan algebra, then the weak central polynomial is of degree $\leq n(m^2-m+1)$.
\end{theorem}

\begin{proof}
By Lemma \ref{dimensions of multiplication} the algebra $\text{End}_F(V)$ of $F$-endomorphisms of $V$ is of dimension $\leq (2m)^2$.
Since ${\mathcal M}(A)$ is a vector subspace of $\text{End}_F(V)$, we obtain that $\dim_F({\mathcal M}(A))\leq (2m)^2$.
By Lemma \ref{dimensions of multiplication} and Corollary \ref{central polynomials for multiplication algebra}
we can apply Theorem \ref{main theorem} for $\dim(R)\leq 4m^2$ and $\dim(V)\geq m$. Hence the pair $(R,V)$ has a weak central polynomial $c'(x_1,\ldots,x_N)$
such that $\deg(c')\leq n(4m^2-m+1)$. If $A$ is a Lie or a Jordan algebra, the vector spaces ${\mathcal L}(A)$ and ${\mathcal R}(A)$ coincide and
the vector space of the multiplications of $A$ is of dimension $m$. Hence $\dim({\mathcal M}(A))\leq m^2$ and this gives that the pair $(R,V)$ has a weak central polynomial of degree $\leq n(m^2-m+1)$.
\end{proof}

\begin{remark}\label{work for Lie algebras}
The arguments of Lemma \ref{multiplication algebra of simple algebra} and Theorem \ref{main corollary}
hold when $\rho:L\to \text{End}_F(W)$ is any non-trivial finite-dimensional irreducible representation of the simple Lie algebra $L$,
$V=\rho(L)\subseteq\text{End}_F(W)$ and $R$ is the subalgebra of $\text{End}_F(W)$ generated by $V$. Hence the pair $(R,V)$ has a weak central polynomial.
\end{remark}

\section{Examples}
In this section we shall use our Algorithm \ref{construction of central polynomial} to construct weak central polynomials for the pairs $({\mathcal M}(L_3)),{\mathcal R}(L_3))$
and $({\mathcal M}(J_2),{\mathcal R}(J_2))$, where $L_3$ is the three-dimensional simple Lie algebra
and $J_2$ is the Jordan algebra of the two-dimensional vector space with non-degenerate symmetric bilinear form. Both algebras ${\mathcal M}(L_3)$ and ${\mathcal M}(J_2)$ are isomorphic to $M_3(F)$.

\subsection{The central polynomial of Formanek}
Our construction uses the central polynomial for $M_k(F)$ discovered by Formanek \cite{F1} in the special case $k=3$. We start with the polynomial
\[
g(u_1,u_2,u_3,u_4)=\sum\alpha_nu_1^{n_1}u_2^{n_2}u_3^{n_3}u_4^{n_4}\in F[u_1,u_2,u_3,u_4], \,\alpha_n\in F,
\]
and define a linear mapping $\theta$ (which is not an algebra homomorphism)
from $F[u_1,u_2,u_3,u_4]$ to the free algebra $F\langle x,y_1,y_2,y_3\rangle$
in the following way:
\[
\theta(g)(x,y_1,y_2,y_3)=\sum\alpha_nx^{n_1}y_1x^{n_2}y_2x^{n_3}y_3x^{n_4}.
\]
Let
\[
g(u_1,u_2,u_3,u_4)=
(u_1-u_2)(u_1-u_3)(u_4-u_2)(u_4-u_3)(u_2-u_3)^2.
\]
Then the polynomial of the free associative algebra
$F\langle x,y_1,y_2,y_3\rangle$
\[
c(x,y_1,y_2,y_k)=\theta(g)(x,y_1,y_2,y_3)+
\theta(g)(x,y_2,y_3,y_1)+\theta(g)(x,y_3,y_1,y_2)
\]
is a central polynomial for the matrix algebra $M_3(F)$
over every field $F$.

\subsection{The three-dimensional simple Lie algebra}
Till the end of the paper we assume that the ground field $F$ is of characteristic different from 2
and $c(x,y_1,y_2,y_3)$ is the central polynomial of Formanek.
We consider the three-dimensional simple Lie algebra $L_3$ with basis $\{e,f,g\}$ and multiplication defined by
\[
[e,f]=g,\,[f,g]=e,\,[g,e]=f.
\]
Then for $v=e,f,g$ the operator $\text{ad}(v):u\to[u,v]$, $u\in L_3$, have the form
\[
\text{ad}(e):(e,f,g)\to (0,-g,f),\text{ad}(f):(e,f,g)\to (g,0,-e),\text{ad}(g):(e,f,g)\to (-f,e,0)
\]
and these three operators span the vector space $\text{ad}(L_3)={\mathcal R}(L_3)$ of the multiplications of $L_3$.
Direct calculations by computer show that
$c(u,v_1,v_2,v_3)=0$ for all $u,v_1,v_2,v_3\in{\mathcal R}(L_3)$. On the other hand
\[
c(\text{ad}(e),\text{ad}(f),\text{ad}(g),\text{ad}(f)\text{ad}(g)):(e,f,g)\to 2(e,f,g),
\]
i.e. $c(x,y_1,y_2,y_3y_4)$ is a weak central polynomial for the pair $({\mathcal M}(L_3),{\mathcal R}(L_3))$.

As we mentioned in the introduction, all weak central polynomials of the pair $({\mathcal M}(L_3),{\mathcal R}(L_3))$ were described in \cite{DoD}.
In particular, the pair has a lot of weak central polynomials which are not obtained from central polynomials of the algebra of multiplications ${\mathcal M}(L_3)$.

\subsection{The three-dimensional simple Jordan algebra}
We fix the basis $\{1,e_1,e_2\}$ of the Jordan algebra $J_2$ of the two-dimensional vector space with non-degenerate symmetric bilinear form
with multiplication defined by
\[
1\cdot e_1=e_1,\,1\cdot e_2=e_2,\,e_1^2=e_2^2=1,\,e_1e_2=0.
\]
Then the linear space ${\mathcal R}(J_2)$ of the multiplications of $J_2$ is spanned by the operators
\[
r(1):(1,e_1,e_2)\to(1,e_1,e_2),r(e_1):(1,e_1,e_2)\to(e_1,1,0),r(e_2):(1,e_1,e_2)\to(e_2,0,1).
\]
As in the case of the Lie algebra $L_3$ the central polynomial $c(x,y_1,y_2,y_3)$ is a weak polynomial identity for the pair $({\mathcal M}(J_2),{\mathcal R}(J_2))$.
But
\[
c(r(e_1+e_2),r(e_1),r(e_1),r(e_1)r(e_2)):(1,e_1,e_2)\to 4(1,e_1,e_2)
\]
and this implies that $c(x,y_1,y_2,y_3y_4)$ is an essential central polynomial for the pair $({\mathcal M}(J_2),{\mathcal R}(J_2))$.

As in the case of the three-dimensional simple Lie algebra, the pair $({\mathcal M}(J_2),{\mathcal R}(J_2))$ has weak central polynomials
which are not central polynomials for the algebra ${\mathcal M}(J_2)\cong M_3(F)$. Such a weak central polynomial of degree 8 is
\[
w(x,y)=[[x,y],x][x,y][[x,y],y]-[[x,y],y][x,y][[x,y],x]+2[x,y]^4.
\]
Direct verification shows that the evaluations of $w(x,y)$ in ${\mathcal R}(J_2)$ are in the center of ${\mathcal M}(J_2)$
and
\[
w(r(e_1),r(e_2)):(1,e_1,e_2)\to 2(1,e_1,e_2),
\]
i.e. $w(x,y)$ is an essential weak central polynomial. By the theorem of Bondari \cite{Bo}, over a field $F$ of characteristic 0 the algebra $M_3(F)$ does not have central polynomials
$c(x,y)$ of degree 8. In the case of a field of arbitrary characteristic direct computations show that
\[
w(e_{13}+e_{33},e_{12}+e_{31})=6(e_{11}+e_{33})-(e_{12}+e_{32}),
\]
where $e_{ij}$ are the usual matrix units, i.e. $w(x,y)$ again is not a central polynomial for $M_3(F)$.

\begin{problem}
It is interesting to find all weak central polynomials for the pair $({\mathcal M}(J_2),{\mathcal R}(J_2))$
in the spirit of the results in \cite{DoD}.
\end{problem}


\begin{thebibliography}{ABC}

\bibitem{A1}
S.A. Amitsur,
On rings with identities,
J. London Math. Soc. {\bf 30} (1955), 464-470.

\bibitem{A2}
S.A. Amitsur,
The $T$-ideals of the free ring,
J. London Math. Soc. {\bf 30} (1955), 470-475.

\bibitem{Bo}
S. Bondari,
Constructing the polynomial identities and central identities of degree $<9$
of $3\times 3$ matrices,
Linear Algebra Appl. {\bf 258} (1997), 233-249.

\bibitem{Bra}
A. Braun,
On Artin's theorem and Azumaya algebras,
J. Algebra {\bf 77} (1982), 323-332.

\bibitem{Bre1}
M. Bre\v{s}ar,
An alternative approach to the structure theory of PI-rings,
Expo. Math. {\bf 29} (2011), No. 1, 159-164.

\bibitem{Bre2}
M. Bre\v{s}ar,
A unified approach to the structure theory of PI-rings and GPI-rings,
Serdica Math. J. {\bf 38} (2012), No. 1-3, 199-210.

\bibitem{DoD}
M. Domokos, V. Drensky,
Cocharacters for the weak polynomial identities of the Lie algebra of $3\times 3$ skew-symmetric matrices,
Adv. Math. {\bf 374} (2020), article 107343.

\bibitem{D1}
V. Drensky,
New central polynomials for the matrix algebra,
Israel J. Math. {\bf 92} (1995), 235-248.

\bibitem{D2}
V. Drensky,
Weak polynomial identities and their applications,
Communications in Mathematics {\bf 29} (2021), No. 2, 291-324.

\bibitem{DF}
V. Drensky, E. Formanek,
Polynomial Identity Rings,
Advanced Courses in Mathematics, CRM Barcelona, Birkh\"auser Verlag, Basel, 2004.

\bibitem{DK}
V. Drensky, A. Kasparian,
A new central polynomial for $3 \times 3$ matrices,
Commun. in Algebra {\bf 13} (1985), 745-752.

\bibitem{DPC}
V. Drensky, G.M. Piacentini Cattaneo,
A central polynomial of low degree for $4 \times 4$ matrices,
J. Algebra {\bf 168} (1994), 469-478.

\bibitem{F1}
E. Formanek,
Central polynomials for matrix rings,
J. Algebra {\bf 23} (1972), 129-132.

\bibitem{F2}
E. Formanek,
Invariants and the ring of generic matrices,
J. Algebra {\bf 89} (1984), 178-223.

\bibitem{F3}
E. Formanek,
A conjecture of Regev about the Capelli polynomial,
J. Algebra {\bf 109} (1987), 93-114.

\bibitem{F4}
E. Formanek,
The Polynomial Iidentities and Invariants of $n \times n$ Matrices,
CBMS Regional Conf. Series in Math. {\bf 78},
Published for the Confer. Board of the Math. Sci. Washington DC,
AMS, Providence RI, 1991.

\bibitem{GSZ}
A. Giambruno, I. Shestakov, M. Zaicev,
Finite-dimensional non-associative algebras and codimension growth,
Adv. in Appl. Math. {\bf 47} (2011), 125-139.

\bibitem{GV}
A. Giambruno, A. Valenti, Central polynomials and matrix invariants,
Israel J. Math. {\bf 96} (1996), 281-297.

\bibitem{GZ}
A. Giambruno, M. Zaicev,
Polynomial Identities and Asymptotic Methods,
Mathematical Surveys and Monographs, {\bf 122},
American Mathematical Society, Providence, RI, 2005.

\bibitem{H}
P. Halpin,
Central and weak identities for matrices,
Commun. in Algebra {\bf 11} (1983), 2237-2248.

\bibitem{J1}
N. Jacobson,
Structure of Rings,
Colloquium Publications. Vol. 37. Providence, R. I.: American Mathematical Society (AMS), 1956.

\bibitem{J2}
N. Jacobson,
PI-algebras. An introduction,
Lecture Notes in Mathematics {\bf 441}, Springer-Verlag, Berlin-New York, 1975.

\bibitem{J3}
N. Jacobson,
Finite-dimensional Division Algebras over Fields,
Berlin: Springer, 1996.

\bibitem{K-BKRo}
A. Kanel-Belov, Ya. Karasik, L.H. Rowen,
Computational Aspects of Polynomial Identities. Volume l: Kemer's Theorems, 2nd edition
Monographs and Research Notes in Mathematics, Boca Raton, FL: CRC Press, 2016.

\bibitem{K-BRo}
A. Kanel-Belov, L.H. Rowen,
Computational Aspects of Polynomial Identities.
Research Notes in Mathematics 9, Wellesley, MA: A K Peters, 2005.

\bibitem{K1}
I. Kaplansky,
Problems in the theory of rings,
Report of a Conference on Linear Algebras, June, 1956,
in National Acad. of Sci.--National Research Council, Washington,
Publ. {\bf 502} (1957), 1-3.

\bibitem{K2}
I. Kaplansky,
Problems in the theory of rings revised,
Amer. Math. Monthly {\bf 77} (1970), 445-454.

\bibitem{Kh}
V.K. Kharchenko,
A remark on central polynomials (Russian),
Mat. Zametki {\bf 26} (1979), 345-346.
Translation: Math. Notes {\bf 26} (1979), p. 665.

\bibitem{LS}
V.N. Latyshev, A.L. Shmelkin,
A certain problem of Kaplansky (Russian),
Algebra i Logika {\bf 8} (1969), 447-448.
Translation: Algebra and Logic {\bf 8} (1969), p. 257.

\bibitem{PoS}
S.V. Polikarpov, I.P. Shestakov,
Nonassociative affine algebras (Russian),
Algebra i Logika {\bf 29} (1990), No. 6, 709-723.
Translation: Algebra and Logic {\bf 29} (1990), No. 6, 458-466.

\bibitem{P}
C. Procesi, Rings with Polynomial Identities,
Marcel Dekker, New York, 1973.

\bibitem{R0}
Yu.P. Razmyslov,
Finite basing of the identities of a matrix algebra
of second order over a field of characteristic zero (Russian),
Algebra i Logika {\bf 12} (1973), No. 1, 83-113.
Translation: Algebra and Logic {\bf 12} (1973), No. 1, 47-63.

\bibitem{R1}
Yu.P. Razmyslov,
On a problem of Kaplansky (Russian),
Izv. Akad. Nauk SSSR, Ser. Mat. {\bf 37} (1973), 483-501.
Translation: Math. USSR, Izv. {\bf 7} (1973), 479-496.

\bibitem{R2}
Yu.P. Razmyslov,
Central polynomials in irreducible representations of a semisimple Lie algebra (Russian),
Mat. Sb. (N.S.) {\bf 122(164)} (1983), No. 1, 97-125.
Translation: Math. USSR, Sb. {\bf 50} (1985), No. 1, 99-124.

\bibitem{Re}
A. Regev,
The polynomial identities of matrices in characteristic zero,
Comm. Algebra {\bf 8} (1980), 1417-1467.

\bibitem{Ro}
L.H. Rowen,
Polynomial Identities of Ring Theory,
Acad. Press, 1980.

\bibitem{W}
J.H.M. Wedderburn, On division algebras,
Trans. Amer. Math. Soc. {\bf 22} (1921), 129-135.

\end{thebibliography}
\end{document}